\documentclass[12pt]{article}
\usepackage{amssymb, url, amsthm, amsmath}
\usepackage{amscd}
\newtheorem{theorem}{\noindent Theorem}
\newtheorem{lemma}{\noindent Lemma}
\newtheorem{definition}{\noindent Definition}
\newtheorem{corollary}{\noindent Corollary}
\newtheorem{conjecture}{\noindent Conjecture}
\newtheorem{statement}{\noindent Proposition}

\date{}

\title{\textbf{Scaling entropy and automorphisms with purely point
spectrum\footnote{Partially supported by the grants
RFBR-08-01-00379-a and RFBR-09-01-12175-ofi-m.}}}

\author{A.~M.~Vershik}

\begin{document}
\maketitle
 \rightline{\it \textbf{To the memory of my friend Misha Birman}}

\begin{abstract}
We study the dynamics of metrics generated by measure-preserving
transformations. We consider sequences of average metrics and
$\epsilon$-entropies of the measure with respect to these metrics.
The main result, which gives a criterion for checking that the
spectrum of a transformation if purely point, is that the {\it scaling
sequence for the $\epsilon$-entropies with respect to the averages of
an admissible metric is bounded if and only if the automorphism has a
purely point spectrum}. This paper is one of a series of papers by the
author devoted to the asymptotic theory of sequences of metric measure
spaces and its applications to ergodic theory.
\end{abstract}

\tableofcontents

\section{Introduction}
Among the many mathematical and nonmathematical problems we have been
discussing
with Misha Birman for many years after our acquaintance began in the
early 1960s, the most intriguing one was the parallel between scattering
theory and ergodic theory. Recently, I have returned to the
(yet nonexistent) ``ergodic scattering theory'' and some forgotten
questions related to it. However, this paper deals with another
subject, which
also correlates with M.~Sh.~Birman's research.

It is well known that the problem of deciding whether or not the spectrum of
a given (say, differential) operator is discrete is quite difficult.
Several remarkable early papers by M.~Sh.~Birman (in the first place, \cite{MB})
dealt with exactly this
problem. In ergodic theory and the theory of dynamical systems, this
problem (whether or not the system of eigenfunctions is complete) is
also very difficult. In what follows, we use the terminology common
in ergodic theory that slightly differs from the one adopted
in operator theory: we use the term ``discrete spectrum'' as a synonym
for ``pure point spectrum.'' This is justified by the fact that
a discrete spectrum in the sense of operator theory almost never appears in the
theory of dynamical systems, since almost always one deals with unitary
operators.

As an example demonstrating the difficulty of this problem, we can
mention the theory of substitutions, or
stationary adic transformations
\cite{V1}, on which an extensive literature exists. The most
intriguing problems concern
nonstationary adic transformations with subexponential
growth of the number of vertices in the corresponding
Bratteli--Vershik diagram (see \cite{V1}). The simplest and most popular
example of such a transformation is the Pascal automorphism defined in \cite{V2}; in this case,
the measure space (i.\,e., the phase space of the dynamical system)
is the space of infinite
paths in the Pascal graph endowed with a Bernoulli measure, and the
transformation sends a path to its successor in the natural
lexicographic order. In spite of many efforts, we still do not know
whether the corresponding unitary operator has a discrete (i.\,e.,
pure point) or mixed spectrum, or, which seems most likely, its spectrum is
purely continuous. Attempts to directly construct its eigenfunctions failed;
another approach, based on the characterization of systems with
discrete spectrum in terms of
Kushnirenko's sequential entropy,
\cite{Ku} has not been carried out.

In the paper \cite{V6} (for more details, see \cite{V3,V4,VG}), we
suggested a new notion, the so-called
\emph{scaling entropy}, which generalizes the notion of Kolmogorov's entropy. The
main point is that we suggest to
{\it average the shifts of the metric} with respect to a
given transformation and then compute the
 {\it $\varepsilon$-entropy} of the average metric.
The class of increasing sequences of positive integers that
normalize the growth of this
$\varepsilon$-entropy over all admissible metrics does not depend on
the choice of an admissible metric, so that the asymptotics of the
growth of these sequences is a new metric invariant of automorphisms. Admissible metrics
are measurable metrics satisfying some conditions that do or do
not depend on the automorphism (see Section~2). It is important that
{\it we consider not merely the $\varepsilon$-entropy of a metric, but
the $\varepsilon$-entropy of a metric in a measure space}. Admissible
metrics play the same role as measurable generating partitions in the
classical theory of Kolmogorov's entropy according to Sinai's definition.
At first sight, the difference between using partitions and metrics
looks rather technical: a partition determines a semimetric of a
special form (a so-called ``cut semimetric,''
see \cite{D}). However,
our approach has two important differences from the classical theory. First,
we use the $\varepsilon$-entropy of the iterated metric on a measure
space rather than the entropy of a partition; this is a
generalization of Kolmogorov's entropy, which allows one to
distinguish automorphisms with zero entropy. The second, more
important, difference is that we use the average metric (rather than the
supremum of metrics, which corresponds to the supremum of partitions),
which has no interpretation in terms of partitions and
which contains more information about the automorphism than the
supremum of metrics.

In this paper, we give a necessary and sufficient condition for an
automorphism to have a discrete spectrum in terms of the scaling
sequence. The condition is that this sequence is bounded. This result
generalizes a theorem of S.~Ferenczi
\cite{Fe,FeP}, who considered the measure-theoretic complexity of
symbolic systems by analogy with the ordinary complexity in symbolic
dynamics. Our approach substantially differs from that of
\cite{Fe}: we consider an arbitrary admissible metric rather than
the Hamming metric only, and, which is most important, introduce the average
metric and show that it is admissible in many cases
and, in particular, for the Hamming metric.

Thus criterion, i.e., the boundedness of the
growth of the scaling entropy, should first be applied to adic transformations,
e.g., to the Pascal automorphism (see \cite{V3}). Although it
is not yet applied to checking that the spectrum of the
Pascal automorphism is not discrete, \footnote{See the footnote in the subsection 7.2},
the corresponding combinatorics
is already developed and described in the recent paper
\cite{MM}, where a lower bound is obtained on the scaling sequence of the sup
metric for the same  Pascal automorphism. Supposedly, one can extend
this bound ($\ln n$) to the average metric using the same techniques.

A wider context is presented in our papers
\cite{V3,V4}, where we suggest a plan for the study of the {\it dynamics
of metrics in a measure space} as a source of new invariants of
automorphisms. It is important that the notion of scaling entropy provides an
answer to the question of whether or not the spectrum of a given
transformation is discrete.

Sections 2--4 are of general nature and are intended not only for the
purposes of this paper, which is devoted mostly to automorphisms with
discrete spectrum. Here we introduce our main objects: admissible
metrics, averages, scaling sequences, and the scaling entropy of an
automorphism.  In Section~5, we study the dynamics of metrics on a group
and find conditions under which the average metric is admissible. The main
result is given in Section~6, where we present a criterion for
checking whether the spectrum of an automorphism is completely
or partially discrete. In the last section, we
sketch possible applications, links to the ordinary construction of
entropy, and general remarks about the dynamics of metrics.

\section{Admissible metrics}

We consider various metrics and semimetrics on a measure space $(X,{\frak A},\mu)$.
In what follows, it is assumed to
be a standard space with continuous measure
$\mu$ and $\sigma$-algebra $\frak A$ of $\bmod 0$ classes of measurable
sets, i.e., a Lebesgue space with continuous measure in the sense of Rokhlin
(see \cite{Ro}). The space $X\times X$ is endowed with the $\sigma$-algebra
$\frak A \times \frak A$ and the measure $\mu \times \mu$.

We define a class of (semi)metrics on a measure space, which plays an
important role in what follows.

\begin{definition}
A measurable function
$\rho:X\times X\rightarrow {\Bbb R}_+$ is called an admissible
(semi)metric if

{\rm 1)} $\rho$ is a (semi)metric in the ordinary sense on a subset
$X'\subset X$ of full measure ($\mu X'=1$), i.e.,
  $\rho(x,y)\geq 0$, $\rho(x,y)=\rho(y,x)$,
  $\rho(x,y)+\rho(y,z)\geq \rho(x,z)$ for all triples $(x,y,z)\in X'\times X' \times X'$,
$(\mu\times\mu)\{(x,x),x\in X'\}=0$, and
$$\int_X\int_X \rho(x,y)d\mu(x)d\mu(y)<\infty.$$

In order to formulate the next condition, observe that if
$\rho$ satisfies Condition~1), then the partition
$\psi_{\rho}$ of the space $X$ into the classes of points
$C_x=\{y:\rho(x,y)=0\}$, $x \in X$, is measurable. Hence
we have a well-defined quotient space
$X_\rho \equiv X/{\psi_{\rho}}$ endowed with the quotient metric denoted by the same
letter $\rho$ and the quotient measure
$\mu_{\psi_{\rho}}\equiv\mu_{\rho}$.
For the quotient space $(X_\rho,\rho,\mu_{\rho})$, Condition~1) is
still satisfied.

{\rm 2)} The (completion of the) metric space $(X,\rho)$ (if $\rho$ is a metric) or
$(X_{\rho},\rho)$ (if $\rho$ on $X$ is a semimetric) is
a Polish (= metric, separable, complete) space
with a Borel probability measure
$\mu$ (respectively, $\mu_{\rho}$).
\end{definition}

Following the measure-theoretic tradition, we must identify
(semi)metrics (and hence the corresponding spaces) if they coincide
almost everywhere as measurable functions on the space
$(X\times X,\mu \times \mu)$. Of course, a (semi)metric that coincides
almost everywhere with an admissible (semi)metric is admissible.\footnote{We can define the notion of {\it almost metric} as
a measurable function for which all axioms on metric satisfy for almost all pairs or triples points.
As F.Petrov noticed for each almost metric there exists the admissible metric in our sense which almost everywhere coincided with it.}

Condition {\rm 2)} means that the $\sigma$-algebra of Borel sets in
the metric space
$(X, \rho)$ (or $(X_{\rho}, \rho)$ in the case of a semimetric) is
dense in the $\sigma$-algebra $\frak A$ of all measurable sets
and, therefore, the measure
$\mu$ (respectively, $\mu_{\rho}$) is a Borel probability measure. It
is obvious from the definition that a semimetric
$\rho$ is admissible if and only if the corresponding metric in the
quotient space $X_{\rho}$ is admissible.

An equivalent definition of an admissible metric is as follows: almost every pair
of points can be separated by balls of positive measure, or, in other
words, the
$\sigma$-subalgebra generated by the open balls  $\bmod 0$
separates points of the
space.

One can also formulate the admissibility condition in terms of the
notion of a pure function from our paper
\cite{VCl}:

\begin{lemma} A metric $\rho$ is admissible if and only if it
satisfies Condition~1) and, regarded as a function of two variables,
is pure in the sense of \cite{VCl}; the latter means that  the partition of the space
 on the classes of equivalence $x\sim y \Leftrightarrow  \mu\{z: \rho(x,z)= \rho(y,z)\}=1$
is the partition on the separate points $mod 0$. In other words: the map $x\mapsto \rho(x,.)$
is injective from measure space $(X,\mu)$ to the classes of $mod 0$ equal functions of one variable.
\footnote{In other words, almost every point is uniquely determined by
the collection of distances from this point to all points of some
 set of full measure (which may depend on the point).}

If $\rho$ is a semimetric, then this condition must hold for the
metric $\rho$ on the quotient space
$(X_{\rho},\rho,\mu_{\rho})$.
\end{lemma}

Indeed, the purity condition implies that the $\sigma$-algebra of sets
generated by the balls separates points and hence is dense in the
$\sigma$-algebra $(X,{\frak A},\mu)$; this also implies the
separability. The converse immediately follows from the
definition of an admissible metric.

It is well known (see \cite{Ro}) that if $(X,\rho)$ is a Polish space,
then every nondegenerate Borel probability measure
$\mu$ on $X$ turns $(X,\rho)$ into a Lebesgue space. In other words,
the metric $\rho$ on a Polish space $(X,\rho)$ endowed with a Borel
probability measure $\mu$ is an admissible metric on the space
$(X,\mu)$.

As in other our papers, in the definition of an admissible metric we reverse
the tradition and {\it consider various metrics and semimetrics on a fixed measure
space rather than various measures on a given metric
space}. Recall that triples
$(X,\rho, \mu)$, consisting of a metric space endowed with a measure,
in  M.~Gromov's book \cite{Gr} were called $mm$-spaces, and in the paper
\cite{VU}, metric triples or Gromov triples.

It is useful to regard admissible metrics as
{\it densities of some finite measures equivalent to the measure
$\mu\times \mu$ on $X\times X$}:
$$dM_{\mu,\rho}=\rho(x,y)d\mu(x)d\mu(y).$$

If we set $\int_{X\times X} \rho(x,y)d\mu(x)d\mu(y)=1$,
which can be done by normalizing the metric, then the new measure is
also a probability measure. Observe the following important property
of admissible metrics implied by this interpretation.

\begin{theorem}
For almost every point $x\in X$, there is a uniquely defined $\bmod 0$
conditional measure $\mu^x$ on $X$, which is given by the formula $d\mu^x(A)=\int_A\rho(x,y)d\mu(y)$
for every measurable set $A\subset X$.
The family of measures $\{\mu^x; x\in X\}$ satisfies the condition
$$\mu(A)=\int_X \mu^x(A)d\mu(x).$$ The metric $\rho$, regarded as a
metric on the measure space $(X,\mu^x)$, is admissible.
\end{theorem}
\begin{proof}
Consider the new measure $M_{\mu,\rho}$ on the space $X\times X$
and the measurable partition into the classes of points
$C^x=\{(x,*)\}\subset X\times X$, and use the classical theorem on the
existence of conditional measures (see, e.g., \cite{Ro}), which implies
the desired formula and the uniqueness $\bmod 0$ of the family of
conditional measures. Now consider the space
 $(X,\rho,\mu^x)$ for a fixed $x$; the metric $\rho$
on this space is admissible since the
admissibility of a metric is obviously preserved under replacing a
measure with an
equivalent one.
\end{proof}

The conditional measure $\mu^x(A)$ can be interpreted as the ``average
distance,'' or the conditional expectation of the distance from the set
$A$ to the point $x$.

In these terms, the condition that the metric, regarded as a function of two variables, is pure
means that the conditional
expectations do not coincide $\mod 0$ for almost all pairs of points.

Now we can give a convenient criterion of admissibility.

\begin{statement}
A measurable function $\rho(\cdot,\cdot)$ satisfying Condition~1) of
Definition~1 is admissible if and only if the following non-degeneracy
condition holds:

for $\mu$- almost all $x$, the measure of arbitrary balls of positive radius is positive:
$\forall \varepsilon >0, \quad \mu\{y:\rho(x,y)<\epsilon\}>0$, which is equivalent to: $$\mu^x([0,\epsilon])>0$$
for almost all  $x$ and all positive $\epsilon$.
\end{statement}
\begin{proof}
The fact that an admissible metric satisfies the condition in question was
observed above. Now assume that this condition is satisfied. We must prove
that the space $(X,\rho)$ (if $\rho$ is a metric) or the quotient space
$X_{\rho}$ (if $\rho$ is a semimetric) is separable. It suffices to
consider the case of a metric.

The condition stated in the proposition implies that for every
$\epsilon$ there exists $\delta=\delta(\epsilon)$ with
$\lim_{\epsilon \to 0} \delta(\epsilon)=0$ such that some set of
measure $>1-\delta$ contains a finite $\epsilon$-net. But this means
exactly that the space is separable and the measure is concentrated on
a $\sigma$-compact set.
\end{proof}

\textbf{Several examples.}
1.~An important example is the following class of semimetrics, which
was in fact intensively used in entropy theory. Every partition
$\xi$ of a space $(X,\mu)$ into finitely or countably many
measurable sets gives rise to a semimetric:
$$\rho_{\xi}(x,y)=\delta_{(\xi(x),\xi(y))},$$
where $\xi(z)$ is the element of $\xi$ that contains $z$. In this
case, $X_\rho$ is a finite or countable metric space. Such
(finite) semimetrics are called {\it cuts}, and their linear
combinations are called {\it cut semimetrics}
(in the terminology of
\cite{D}). It is easy to see that cut (semi)metrics are admissible.

\medskip
2.~The very important metric defined by the formula $\rho(x,y)={\rm const}$ for $x \ne y$
determines a discrete uncountable space. We will call it the constant
metric. {\it The constant metric on a space with continuous measure is
not admissible.} In this case, the $\sigma$-algebra
generated by the open sets is trivial and does not separate points of the
space, i.e., a point is not determined by the collection of
distances to the other points.

\medskip
3.~The condition defining an admissible metric can be strengthened by
requiring, in condition~1), that
$$\int_X\int_X \rho(x,y)^p d\mu(x)d\mu(y)<\infty$$
for $p>1$; in this case, we say that the metric is
{\it $p$-admissible}.  Let us say that $\infty$-admissible metrics
(semimetrics) are {\it bounded}; this class of admissible metrics will
be most
useful in what follows.

\medskip
4.~In combinatorial examples, it often suffices to consider
metrics with which the space is compact or precompact (or, in the case of a
semimetric, quasi-compact).
For example, adic transformations \cite{V1,V2} act in the space of
infinite paths of an $\Bbb N$-graded graph, which is a totally
disconnected compact space.

\medskip
The set of admissible metrics on a Lebesgue space
$(X,\mu)$ with continuous measure is a convex cone
${\cal R}(X,\mu)={\cal R}$ with respect to the operation of taking a linear combination
of metrics with nonnegative coefficients. This cone is a canonical
object (by the uniqueness of a Lebesgue space up to isomorphism) and plays a
role similar to the role of the simplex of Borel probability measures in
topological dynamics. In many cases, it suffices to consider
admissible (semi)metrics that produce compact spaces (after completion), but we
do not exclude the case of a noncompact space. One may consider
different topologies on the cone
${\cal R}$; the most natural of them is the weak topology, in which an
$\varepsilon$-neighborhood of a metric $\rho$ is the collection of
metrics
$$
\Bigl\{\theta:
(\mu\times\mu)\{(x,y):|\rho(x,y)-\theta(x,y)|<\varepsilon\}>1-\varepsilon\Bigr\}.
$$

The property of being an admissible metric is invariant under
measure-preserving transformations:
if a metric $\rho$ is admissible and a transformation $T$ of the space
  $(X,\mu)$ preserves the measure  $\mu$, then the image  $\rho^T$ of $\rho$,
defined by the formula $\rho^T(\cdot,\cdot)=\rho(T\cdot,T\cdot)$, is also admissible.
Thus the group of (classes of) measure-preserving transformations acts
on the cone ${\cal R}$ in a natural way.

\section{Average and maximal metrics}

Let $T$ be a measure-preserving transformation (in what follows, it
will be
an automorphism). When considering automorphisms or groups of
automorphisms in spaces with admissible semimetrics, it is natural to
assume that there exist an invariant set of full measure on which the
(semi)metric is admissible in the sense of our definition.

In addition to the notion of admissible (semi)metrics, we define the
class of  {\it $T$-admissible metrics}. The $T$-admissibility
condition must be invariant, in the sense that if
$\cal M$ is the class of $T$-admissible metrics, then
 $$V{\cal M}\equiv \{\rho: \rho(x,y)=\rho_1(Vx,Vy),\;\rho_1 \in \cal
 M\}$$ is the class of
$VTV^{-1}$-admissible metrics in the same sense. Among many possible
versions, we choose the class of admissible (semi)metrics for which
$T$ is a {\it Lipschitz transformation} almost everywhere:
there exists a positive constant $C$ such that for $(\mu\times \mu)$-almost all pairs
$(x,y)$, the condition
$$\rho(Tx,Ty)\leq C\rho(x,y)$$
holds; let us say that metrics from this class are Lipschitz $T$-admissible (semi)metrics.

The choice of an appropriate class depends on the problem under
consideration and the properties of the automorphism. For example, in
the case of
adic automorphisms, it is most convenient to consider the class of Lipschitz metrics.
This class can also be defined for countable groups of automorphisms.

For an arbitrary admissible semimetric $\rho$, we have defined the
partition $\psi_{\rho}$ of the space
$(X,\mu)$. In the same way, given an arbitrary automorphism $T$ of the
space $(X,\mu)$, we consider the $T$-invariant partition
$\psi_{\rho}^T=\bigvee_{k=0}^{\infty} T^k \psi_{\rho}$.
We say that the semimetric $\rho$ is {\it generating} for $T$
if $\psi_{\rho}^T$ is the partition into separate points
$\bmod 0$ (which we denote by $\varepsilon $). If
$\rho$ is the metric generated by a finite partition, then this is the
ordinary definition of a generator   (see \cite{Ro1}).

Let us define the average metric and the $\sup$-metric for a given
automorphism.

\begin{definition}
Let $T$ be an automorphism of a space $(X,\mu)$, and let $\rho$ be an
admissible metric.

The average metric $\rho_n^T$ is defined by the formula
$$\hat \rho_n^T(x,y)=\frac{1}{n}\sum_{k=0}^{n-1}\rho(T^kx,T^ky).$$

The $\sup$-metric is defined by the formula
$$
{\bar \rho}_n^T(x,y)=\sup_{0\le k<n}\rho (T^k x,T^k y).
$$
\end{definition}

The following important result is a direct corollary of the pointwise
ergodic theorem.

\begin{theorem}
For any automorphism $T$ and any admissible (semi)metric
$\rho$, the limit of the sequence of
average (semi)metrics $\rho^T_n$, which we denote by $\hat \rho$,
exists almost everywhere in the space
$(X\times X, \mu \times \mu)$:
$$\hat \rho^T (x,y)=\lim_{n \to \infty}\frac{1}{n}\sum_{k=0}^{n-1}\rho(T^kx,T^ky);$$
$\hat \rho$ is a metric if and only if
$\rho$ is a metric or a generating semimetric.
\end{theorem}

The existence of the limit follows from the fact that the integral
$\int_X\int_X\rho(x,y)d\mu(x)d\mu(y)$ is finite and the ergodic theorem.

\begin{definition}
The metric
$$
\hat \rho^T(x,y)=\lim_{n\to \infty} \frac{1}{n}\sum_{k=0}^{n-1}\rho(T^kx,T^ky)
$$
  (the limit exists
$\mu$-a.e.) is called the {\it average}, or the $l^1$-average of $\rho$
with respect to the automorphism $T$.

The metric
$$\bar \rho^T(x,y)=\sup_{k\geq 0}  \rho(T^k x, T^k x)$$
is called the limiting $\sup$-metric of $\rho$
with respect to $T$.
\end{definition}

It is clear that $\hat \rho^T$ and $\bar \rho^T$
satisfy all conditions of the definition of a (semi)metric;
in what follows, we will
consider only admissible metrics or generating semimetrics
$\rho$, so that
$\hat \rho^T$, and  $\bar \rho^T$ for an ergodic automorphism $ T$,  are
metrics. The superscript $T$ in the notation for the average and
$\sup$-metrics will be omitted if the automorphism is clear from the
context. It is also obvious that
$\hat \rho \leq \bar \rho$. However, the metrics
$\hat \rho, \bar \rho$ may not be admissible even if the
(semi)metric $\rho$ is admissible.

In what follows, we will mainly consider the
{\it average metric} $\hat \rho$. In most interesting cases, namely
if $T$ is  weakly mixing, i.e., its spectrum in the orthogonal complement
to the space of constants is continuous, this metric is constant and
hence not admissible; the $\sup$-metric may be constant even
for automorphisms with discrete spectrum. However, in all cases we
will be interested not in the limiting metrics themselves, but in the
{\it asymptotic behavior of the averages
$\rho_n^T$ as $n\to \infty$}. As we will see, automorphisms with
discrete spectrum stand apart, since for them the average metric is
often admissible.

It is clear that $\hat \rho$ is nothing else but the projection of the
function $\rho$, regarded as an element of the space
$L^1(X\times X,\mu \times\mu)$, to the subspace of $(T\times
T)$-invariant functions, i.e., the expectation of the metric
$\rho$ with respect to the subspace of invariant functions on
$X\times X$. This space consists of constants if and only if $T$ has
no nontrivial eigenfunctions (in other words, $T$ is weakly mixing).
In this case,
$\hat \rho$ is almost everywhere a constant, which is equal to the average
$\rho$-distance between the points of the space $X$. At the same time, if
$T$ is not weakly mixing, then the spectrum of $T$ contains
a discrete component and $\hat \rho$ may be a nonconstant
{\it $T$-invariant (semi)metric}. In this case, one may obtain bounds
on  $\hat \rho$ using
Fourier analysis.

Definitions and Lemma~1 imply the following lemma.

\begin{lemma}
Let $\rho$ be an admissible Lipschitz metric for an automorphism $T$;
then the average metric $\hat \rho$ is also admissible and Lipschitz
for $T$.
\end{lemma}

We do not prove this assertion, because we do not need it in this paper.

Here is an example of computing the average metric generated by the
cut semimetric in the case of  a rotation of the circle.

\medskip
\textbf{Example}. Let $X={\Bbb T}^1={\Bbb R}/\Bbb Z$, and let
$\lambda\in X$ be an irrational number. Consider the semimetric
$\rho(x,y)=|\chi_A(x)-\chi_A(y)|$, where $\chi_A$ is the indicator of
a measurable set $A\subset X$; the metric
$\rho(x,y)$ is $T$-admissible in the sense of our definition for the
shift $T_{\lambda}$ by any irrational number
  $\lambda$. The corresponding average metric is shift-invariant
and looks as follows:
$\hat\rho(x,y)=m[(A+x)\Delta\mu(A+y)]=m[A\Delta (A+x-y)]$; it is
obviously admissible.

\section{Entropy and scaling entropy}

\subsection{The $\varepsilon$-entropy of a measure in a metric space}

Recall that the $\varepsilon$-entropy of a compact metric space
$(X,\rho)$ is the function  $\varepsilon\mapsto
H_{\rho}(\varepsilon)$ whose value is equal to the minimum number of
points in an $\varepsilon$-net of  $(X,\rho)$.
\begin{definition}

The $\varepsilon$-entropy of a measure space
$(X,\mu)$ with an admissible metric
$\rho$ is the function
$$\varepsilon\mapsto H_{\rho,\varepsilon}(\mu)=\inf \{H(\nu): k_{\rho}(\nu,\mu)<\varepsilon\},$$
where $\nu$ ranges over the set of all discrete measures with finite
entropy and $k_{\rho}(\cdot,\cdot)$ is the Kantorovich metric on the
space of Borel probability measures on $(X,\rho)$. The entropy of a
discrete measure $\nu=\sum_i c_i \delta_{x_i}$ is defined in the usual
way: $H(\nu)=-\sum_i c_i\ln c_i$.
\end{definition}

For our purposes, it is more convenient to use
in the above definition another characteristic instead of
$H(\mu)$, namely,
$$
H'_{\rho,\varepsilon}(\mu)=
\min\Bigl\{\ln k: \exists X', \mu(X')>1-\varepsilon,\;
\exists \{x_i\}_1^k: X'\subset \bigcup_{i=1}^k V_{\varepsilon}(x_i)\Bigr\},
$$
where $V_{\varepsilon}(x)$ is the $\varepsilon$-ball centered at $x$;
thus $\{x_1,x_2,\dots, x_k\}$ is an $\varepsilon$-net in
$X'$. The finiteness of $H'$ follows from the fact that a Borel
probability measure in a Polish space is concentrated, up to
$\varepsilon$, on a compact set; $H'$ is more convenient for
computations than $H$.

In this paper, we use the following simple inequality.

\begin{lemma}
For every compact metric space
$(X,\rho)$ and every nondegenerate Borel measure
$\mu$, the following inequality holds:
$$H_{\rho,(d+1)\varepsilon}(\mu)\leq H'_{\rho,\varepsilon}(\mu),$$
where $d$ is the diameter of the space.
\end{lemma}
\begin{proof}
Assume that the diameter of the compact space does not exceed $1$. Assume that the
measure of a set
$X'$ is greater than $1-\varepsilon$ and $X'\subset
\bigcup_{i=1}^k V_{\varepsilon}(x_i)$. Thus the points
$x_1,\dots, x_k$ form an $\varepsilon$-net in $X'$.
Consider the discrete measure $\nu$ supported by the points $x_1,\dots,
x_k$ with charges $\nu(x_i)$ equal to the measures $\mu(V(x_i))$ of the corresponding balls
(if two balls have a nonempty intersection, then
we distribute the
measure of the intersection proportionally between their centers). Choose an arbitrary point
$x_{\infty}$ and set its charge equal to
$1-\mu(X')$. Now consider the Monge--Kantorovich transportation problem
with input measure
$\mu$ and output measure $\nu$.
It is easy to see that we have in fact
determined an admissible plan $\Psi$ for this  problem: the
transportation from a point $x\in X'$ goes to the points $x_i$ for
which $x\in V_{\varepsilon}(x_i)$, and
the remaining part of the measure $\mu$ on
the set
$X\setminus X'$ goes to the point $x_{\infty}$.
It is easy to compute
the cost of this plan; this gives a bound on the Kantorovich distance
between the measures $\nu$ and $\mu$:
$$k_{\rho}(\nu,\mu)\leq\varepsilon(1-\varepsilon)+\varepsilon<2\varepsilon.$$
On the other hand, we have $H(\nu)\leq \ln k
=H'_{\rho,\varepsilon}(\mu)$.
\end{proof}

\subsection{Scaling sequence and scaling entropy}

Let us define the notion of a scaling sequence for the entropy of an
automorphism. If an automorphism $T$ is fixed, we omit the superscript
in the notation for the average entropy and white simply
$\hat\rho_n$.

\begin{definition}
Let $T$  be an automorphism of a Lebesgue space
$(X,\mu)$ with a
$T$-invariant measure $\mu$. By definition, the class of scaling
sequences for the automorphism $T$ and a given (semi)metric
$\rho$ on $X$ is the class, denoted by
${\cal
H}_{\rho,\varepsilon}(T)$, of increasing sequences of positive numbers
$\{c_n, n\in {\Bbb N}\}$ such that
$${\cal H}_{\rho,\varepsilon}(T)= \Bigl\{\{c_n\}: 0<\liminf_{n\to \infty}\frac{H_{\hat\rho_n,\varepsilon}(\mu)}{c_n}\leq
\limsup_{n\to \infty}\frac{H_{\hat \rho_n,\varepsilon}(\mu)}{c_n}
<\infty\Bigr\}.$$
\end{definition}

In many cases, the class of scaling sequences for a given metric
$\rho$ does not depend on sufficiently small $\varepsilon$. In this
case, it is obvious that all sequences
$\{c_n\}$ from ${\cal
H}_{\rho}(T)$ are equivalent.

\begin{definition}
Assume that for a given ergodic automorphism $T$ of a space
$(X,\mu)$ there exists a (semi)metric
$\rho_0$ such that the class of scaling sequences for $\rho_0$ is the
maximal one
(i.e., for any other (semi)metric, sequences
$\{c_n\}$ from the corresponding class grow not faster than for
$\rho$). In symbols, we write this fact as
$${\cal H}_{\rho_{0}}(T)=\sup_{\rho}{\cal H}_{\rho}(T).$$
Then we say that ${\cal H}_{\rho_{0}}(T)$ is the class of scaling sequences for the
automorphism $T$ and the metric $\rho_0$ is $T$-maximal.
\end{definition}

It seems that such a metric exists for every automorphism. If we have
chosen some $T$-maximal scaling sequence and the corresponding limit
of entropies does exist, then it is called the
{\it scaling entropy}.

\begin{conjecture}
For every ergodic automorphism $T$, a generic $T$-admissible Lipschitz
metric is $T$-maximal. In particular, for a $K$-automorphism (i.e., an
automorphism with completely positive entropy), the scaling sequence is
equivalent to the sequence
$c_n=h(T)n$, where $h(T)$ is the entropy of $T$, for every Lipschitz
metric.
\end{conjecture}

A preparatory result in this direction was obtained in
\cite{Ke}.

In this paper, we will prove that for an automorphism with purely discrete
spectrum and a $T$-admissible metric,  the class ${\cal H}_{\rho}(T)$ of scaling sequences
  is the class of bounded sequences.

\section{Invariant metrics on groups and averages of admissible
metrics}

\subsection{Invariant metrics and discrete spectrum}

Let us recall some known facts about ergodic automorphisms with discrete
spectrum. It obviously follows from the character theory of
commutative groups that the spectrum of a translation on a compact Abelian
group is discrete. By the classical von Neumann theorem, the converse
is also true: an ergodic automorphism with discrete spectrum is
metrically isomorphic to the translation $T$ on a compact Abelian group $G$
endowed with the Haar measure $m$ by an element whose powers form a
dense subgroup:
$$x\mapsto Tx=x+g,\quad \operatorname{Cl}\{ng,\, n\in \Bbb Z\}=G$$
(we use the additive notation). Note that on a compact group $G$ there
are many metrics that are invariant under the whole group of
translations
and determine the standard group topology. We will need the
following assertion (which is, possibly, partially known).

\begin{statement}
The spectrum of an ergodic automorphism $T$ of a measure space
$(X,\mu)$ for which there exists a $T$-invariant admissible semimetric
$\rho$ contains a discrete component. Moreover, if
$\rho$ is a metric, then the spectrum of  $T$ is discrete and,
consequently, $T$ is isomorphic to a translation on a compact Abelian group.
\end{statement}

\begin{proof}
Since an admissible (semi)metric lies in the space
$L^1(X\times X, \mu \times \mu)$, the tensor square of the operator $U_T$,
which corresponds to the automorphism $T\times T$, has nonconstant
eigenfunctions. This can happen only if the spectrum of the unitary
operator $U_T$ contains a discrete component, and the first claim is
proved. In other words, $T$ is an extension of some
quotient automorphism with discrete spectrum, which may coincide
with $T$ itself. This means that $T$ is a skew product over a base
with discrete spectrum.
Denote by $H \subset L^2(X,\mu)$ the subspace spanned by all
eigenfunctions of $U_T$. All invariant functions of the operator
$U_T\otimes U_T$ belong to the tensor square $H\otimes H$;
hence these functions, regarded as functions of two variables, do not change
when the argument ranges over a fiber of the skew product. But if these
fibers are not single point sets, it follows that
the metric does not distinguish points in fibers and hence is a
semimetric. Thus if  $\rho$ is a metric, then each fiber
necessarily consists of a single point, the spectrum of $T$ is purely
discrete, and, by the ergodicity, $T$ is isomorphic to a translation on a
group.
 \end{proof}

Let us supplement this proof with an important refinement. Assume that the
automorphism $T$ has an orbit that is everywhere dense with respect to
the metric $\rho$, i.e., $T$ is topologically transitive (though we do
not assume that it is a priori continuous). It is clear from above that
this condition follows in fact from the existence of an invariant
metric. Since the metric is admissible, we may assume without loss of
generality that $(X,\rho)$ is a Polish space. Consider the dense orbit
$O=\{T^nx,\, n \in \Bbb Z\}$ of some point $x$. The restriction of $\rho$
to $O$ is a translation-invariant metric on the group
$\Bbb Z$, and all translations are isometries. Hence the completion of
$O$ is an Abelian group to which we can extend the translations and
their limits. Therefore $X$ is a Polish monothetic
group.\footnote{A topological group that contains a dense infinite cyclic
subgroup is called monothetic. Note that there are many non-locally
compact monothetic groups on which there is an invariant metric, but
there is no invariant measure. A recent example is the Urysohn universal
space regarded as a commutative group, see \cite{CV}.}
Obviously, the measure $\mu$ is invariant under the action of the
closures of powers of $T$, i.e., it is an invariant probability
measure on the whole group. Hence, by Weil's theorem, $X$ is a compact
commutative group and $T$ is the translation by an element whose powers are
everywhere dense.

\subsection{Admissible invariant metrics}

As we have already observed, for weakly mixing automorphisms, the
average of every metric is constant, since there are no
other invariant metrics. However, for automorphims whose spectra
contain a discrete component or are purely discrete, there are many
invariant (semi)metrics. Hence, in order to study such automorphisms,
we should investigate the question when the average  metric
for an automorphism $T$ with a discrete spectrum is admissible.

Given a translation-invariant metric
$\rho$ on a compact commutative group with the Haar measure, consider
the function
 $\phi_{\rho}(r)=\rho(x,x+r)=\rho(0,r)$. In the example from Section~3,
 it looked as $\phi(r)\equiv\hat\rho(x,x+r)= m[A\Delta (A+r)]$.

One can easily write down necessary and sufficient algebraic conditions on
a measurable function $\phi$ that guarantee that it can be written as
$\phi_{\rho}$ for a measurable invariant metrics:
 $$\phi(\textbf{0})=0,\;\phi(x)\geq 0, \;\phi(-x)=\phi(x),\; \phi(x)+\phi(y)\geq \phi(x+y).$$

We will not need these conditions; the only important fact is that the
admissibility condition from Lemma~1 can easily be reformulated in
terms of this function.

\begin{theorem}
An invariant measurable (semi)metric on a commutative compact group
(satisfying Condition~1 from Definition~1) is admissible if and only if
any of the following conditions holds.

{\rm1.} The corresponding function $\phi_{\rho}$ is measurable, and
$$\mu \{z: \phi_{\rho}(z)\ne \phi_{\rho}(g+z)\}>0$$
for almost all $g$.

{\rm2.} $$\operatorname{ess\,inf}_{g\in V\setminus 0}\phi_{\rho}(g)= 0,$$
where $V$ is an arbitrary neighborhood of the zero of the group in the
standard topology.
\end{theorem}

\begin{proof}
Condition~1 is exactly equivalent to the condition
of Lemma~1. It is useful to give a direct proof that this condition is
necessary. Assume that it is not satisfied, i.e., for all
elements $g$ from some set of positive Haar measure,
$\phi_{\rho}(z)=\phi_{\rho}(g+z)$ for almost all $z$. Therefore, the
measurable function $\phi$ is constant on cosets of the subgroup
generated by $g$. However, every set of positive Haar measure in a
nondiscrete Abelian compact group on which there is an ergodic
translation
contains an element $g$ that generates a dense cyclic subgroup. Then,
since the function $\phi_{\rho}$ is measurable,
it follows that it is constant almost everywhere. But if the function
$\phi_{\rho}$ is constant, then the metric $\rho$ is also constant
and hence not admissible.

Let us prove that Condition~2 of the theorem is necessary. Consider
the function $\phi_{\rho}$ for an invariant admissible metric
$\rho$. Assume that $\theta=\operatorname{ess\,inf}$ from Condition~2 is positive; then, using the
invariance of $\rho$, we can construct a continuum of points lying
at a fixed positive $\rho$-distance greater than $\theta$,
which contradicts the admissibility. For the same reason, the set of values
of  $\phi_{\rho}$ is dense in some neighborhood of the
zero. The fact that Condition~2 is sufficient follows from Proposition~1.
\end{proof}

The last assertion implies the following corollary.

\begin{corollary}
For every admissible invariant metric there exists a sequence of
group elements that converges to zero both in the standard topology and
with respect to the (semi)metric.
\end{corollary}

\subsection{Admissibility of the average metric}

Now we can explicitly write down the condition that guarantees the admissibility of the
average, i.e., invariant, metric in terms of the original metric.
Consider an arbitrary measurable (non necessarily admissible) metric
$\rho$ on a compact commutative group $G$ and an ergodic translation $T$ by
an element $g$. Let us write down an expression for the average metric:
\begin{eqnarray*}
{\hat \rho}^T(x,y)&=&\lim_{n\to \infty} \frac{1}{n}\sum_{k=0}^n \rho
(x+kg,y+kg)=\int_G \rho(g+x,g+y)dm(g)\\
&=&\int_G \rho(z,z+y-x)dm(z)
\end{eqnarray*}
(we have used the fact that the measure $m$ is invariant). Hence the
function
$\phi(r)$, regarded as a function of $r$, is measurable and has the form
$$ \phi(r)=\int_G \rho(z,z+r)dm(z)=\hat \rho(x,y);\quad y-x=r.$$

Obviously, $\hat \rho$ is a measurable function on the group
$G\times G$. Now we can check the admissibility condition for the
average metric $\hat \rho$.

\begin{definition}
We say that an admissible metric $\rho$ on a compact commutative group
$G$ with Haar measure $m$ is semicontinuous at zero in mean if
$$\liminf_{r\to 0}\int_G \rho(x,x+r)dm(x)=0;$$
we say that it is semicontinuous at zero in measure if the
following condition holds (in which ``meas'' means convergence in
measure):
$$\liminf_{r\to 0}(meas) \rho(x,x+r)=0.$$
\end{definition}

Note that the second condition follows from the first one, and that
both conditions are stated in purely group terms, i.e., do not depend
on the particular translation $T$.

Thus we have the following admissibility criterion for the average
metric.

\begin{statement}
The average metric $\hat\rho^T$ for an ergodic translation $T$ on a compact
commutative group $G$ is admissible if and only if the original
(semi)metric $\rho$ is admissible and semicontinuous at
zero in mean.
\end{statement}

\begin{proof}
The ``if'' part is proved above; the ``only if'' part follows from the
equality
$\phi_{\hat\rho}(r)=\int_G \rho(x,x+r)dm(x)$ and the previous
proposition.
\end{proof}

Now we are ready to formulate and prove the following important fact.

\begin{theorem}
For every bounded admissible metric on a compact
commutative group $G$, the average metric is admissible.
\end{theorem}

\begin{proof}
Assume that the metric $\rho$ is not semicontinuous at zero and there
exists a positive number
$c>0$ such that $$\liminf_{r\to 0}\int_G \rho(x,x+r)dm(x)> c.$$
Assume that the metric is normalized so that the diameter of the space $X$ is equal to 1. Then it
follows from our assumption that for sufficiently small (i.e.,
belonging to a small neighborhood of the zero in the group $G$) $r$ there
exists a (depending on $r$) subset in $X$ of measure $\alpha$,
which does not depend on $r$ and is greater than
$\frac{c}{2}$, on which $\rho(x,x+r)>\frac{c}{2}$.
Indeed, $1 \cdot  \alpha +(1-\alpha)\frac{c}{2}>\int_G \rho(x,x+r)dm(x)> c$.
But since the group is compact, there is a set of positive
measure for all points $x$ of which the inequality
$\rho(x,x+r)>\frac{c}{2}$ holds for all sufficiently small $r$ from
some set of positive measure. This in turn contradicts the
admissibility of the metric $\rho$ in the formulation of
Proposition~1: arbitrarily small values
$\phi(r)$ for small $r$ cannot interlace with values greater than
$\frac{c}{2}$ by the triangle inequality. Thus we have proved that
$\rho$ is semicontinuous.
\end{proof}

\medskip\noindent{\bf Question.}
Does there exist an unbounded admissible metric
$\rho$ on the circle $S^1={\Bbb R}/{\Bbb Z}$ for which the average metric
  $\hat \rho$ is not admissible? Does there exist an unbounded
$p$-admissible metric, with $1<p<\infty$, on a compact
Abelian group for which the average metric is not
admissible?\footnote{While the paper was in press, F.~Petrov and
P.~Zatitskiy proved that the average of any (in
particular, unbounded) admissible metric on the circle
is admissible. Thus the additional
assumptions
that the average metric is admissible in Proposition~2 and Theorem~5
are superfluous.}

\section{Criterion for the discreteness and continuity of the spectrum in
terms of the scaling entropy}

Now we formulate our main result.

\begin{theorem}
For an ergodic automorphism with discrete spectrum realized as a
translation
on a compact commutative group with an arbitrary bounded admissible
metric, or, more generally, with a metric for which the average metric
is admissible, the scaling sequence is bounded.
\end{theorem}

\begin{proof}
Since the set $B_{\varepsilon}=\phi^{-1}([0,\varepsilon])$
is, by definition, the ball of radius
$\varepsilon$ centered at
$\textbf{0}\in G $ in the metric  $\hat \rho$, which is admissible by
assumption, it follows that $m B_{\varepsilon}>0$, since the metric is
nondegenerate (see the definition of an admissible metric). But the sum
$B_{\varepsilon}+B_{\varepsilon}$, like the sum $A+A$ for every set $A$ of
positive Haar measure in a locally compact group,
contains a neighborhood $V$ of the zero in the standard topology (see, e.g.,
\cite{We}).
It follows from the triangle inequality that
$B_{\varepsilon}+B_{\varepsilon}\subset B_{2\varepsilon}$,
so that $V \subset B_{2\varepsilon}$. Since $\varepsilon>0$  is
arbitrary, we see that
{\it the topology on $G$ determined by the average metric $\hat\rho$
coincides with the standard topology},
i.e.,  in the topology determined by
$\hat \rho$, the group $G$ is compact and contains a finite $\varepsilon$-net for every $\varepsilon$.

The pointwise a.e. convergence
 $$\lim_{n \to \infty} \frac{1}{n}\sum_{k=0}^{n-1} \rho(x+kg,y+kg)={\hat \rho}_T(x,y),$$
which follows from the pointwise
ergodic theorem, implies that the number of points in an $\varepsilon$-net for $(G,\rho_n)$
tends to the number of points in an  $\varepsilon$-net for
$(G,\hat \rho)$. This means that the sequence $H_{\rho_n,\varepsilon}(X)$
converges to
$H_{\hat \rho,\varepsilon}(G)$. From the inequalities of Lemma~3
we see that the sequence
$H_{\rho_n,\varepsilon}(\mu)$ is bounded as $n \to \infty$,
and thus the scaling sequence for the automorphism $T$, which acts on
the metric triple
$(G,\rho, m)$, is bounded.
\end{proof}

Combining this theorem with the previous one, we obtain the following result.

\begin{theorem}
An ergodic automorphism $T$ has a discrete spectrum if and only if
the scaling sequence for $T$ is bounded for some, and hence for every, bounded
admissible metric.
\end{theorem}

\begin{proof}
Above we have proved that if an automorphism has a discrete spectrum
and the average metric is admissible, then the scaling
sequence is bounded. But the average metric is always admissible
provided that the original metric is bounded and admissible.

Assume that the scaling sequence is bounded for an automorphism $T$
and an admissible metric
$\rho$.  Recall that the average metric is indeed a metric (and not a
semimetric). Consequently, the space
$(X,\hat \rho)$ is precompact, and hence the metric
$\hat \rho$ is admissible. Since it is $T$-invariant, it follows from
Theorem~1 that $T$ has a purely discrete spectrum.
\end{proof}

Combining the last theorems
with the previous results yields a criterion for the
discreteness and continuity of the spectrum in terms of the
automorphism $T$ and an arbitrary admissible metric.

\begin{theorem}
Let $T$ be an ergodic automorphism, and let
$\rho$ be a bounded admissible (semi)metric. If the corresponding
scaling sequence is not bounded, then the spectrum of $T$
contains a continuous component. If the
scaling sequence is not bounded for every admissible (semi)metric,
then the spectrum of $T$ is purely continuous.
\end{theorem}

In the next section, we will show how one could apply this criterion.

\section{Comparison with the traditional approach, the Pascal
automorphism, and concluding remarks}

\subsection{Supremum metrics}

The entropy theory of dynamical systems, developed mainly by
Kolmogorov, Sinai, and Rokhlin,  essentially uses the tools of the
theory of measurable partitions. In Sinai's definition, the entropy
appears as an asymptotic invariant of the dynamics of finite
partitions under the automorphism:
 $$\lim_n\frac{H(\bigvee_{k=0}^{n-1} T^k\xi)}n=h(T,\xi).$$

As a result of this theory, the study of the class of automorphisms with completely
positive entropy was differentiated into a separate field, whose methods do
not apply to automorphisms with zero entropy. For example,
one cannot obtain a new invariant for such
automorphisms following the same
scheme. This can be seen from the following simple fact.

\begin{statement}
 For every transformation $T$ and every increasing sequence of positive
numbers $\{c_n,\, n\in \Bbb N\}$ satisfying the condition
$lim_n\frac{c_n}{n}=0$, there exists a generating partition
$\xi$ such that
 $$\lim_n\frac{H(\bigvee_{k=1}^{n} T^k\xi)}{c_n}=\infty.$$
  \end{statement}

This means that the maximum growth of the entropy
$H(\bigvee_{k=1}^{n} T^k\xi)$ either is linear (for automorphisms with
positive entropy), or, in the case it is sublinear, it is arbitrarily close to
linear for every automorphism. Thus we obtain no new information.

The metric corresponding to the supremum (product)
$\xi_n=\bigvee_{k=1}^n T^k\xi$ of partitions is the supremum of the shifted metrics:
${\bar \rho}_n^T(x,y)=\sup_{0\le k<n}\rho (T^k x,T^k y)$.
Hence, following our plan, we can use
the $\varepsilon$-entropy of the metric
${\bar \rho}_n^T(x,y)$ instead of the entropy of the partition
$\xi_n$ itself. Then, using the definitions from Section~4, for a given
metric $\rho$ we can introduce an analog of the function
${\cal H}_{\rho,\varepsilon}(T)$ with the metric $\hat\rho_n$ replaced
by $\bar \rho_n$:
$${ \cal \bar H}_{\rho,\varepsilon}(T)= \Bigl\{\{c_n\}: 0<\liminf_{n\to \infty}\frac{H_{\bar\rho_n,\varepsilon}(\mu)}{c_n}\leq
\limsup_{n\to \infty}\frac{H_{\bar \rho_n,\varepsilon}(\mu)}{c_n}
<\infty\Bigr\}.$$

In this way we define the class of {\it $\sup$-scaling sequences}
$\bar c_n$ for a given metric $\rho$. This also allows us to extend
the classical entropy theory following the above scheme. Though it
is somewhat easier to deal with the sup-metric than with the average
metric, the former is much less useful than the latter.
The metric  $\bar \rho$ more often happens to be constant for an
automorphism with discrete
spectrum, while, as we have seen,   $\hat \rho$ is always admissible if
the original metric is bounded. Let us illustrate the important
difference between the operations of taking the average and
supremum metrics by the following example.

\medskip\noindent
\textbf{Example.} Let $T$ be an irrational rotation of the unit circle,
and let $\rho$ be the semi-metric corresponding to
a generating two-block partition (i.e., a partition into
two sets of positive measure); the semi-metric
$\rho$ is $T$-admissible, hence, as we have seen,
 $\hat \rho$ is an invariant admissible metric. At the same time,
$\bar \rho$ is the constant metric. This means that the scaling sequence
$c_n$ is bounded, but
$\bar c_n$ (the scaling sequence for the sup-metric) is not; namely,
we have
 $\bar c_n\sim \ln n$. Thus the difference manifests itself even in the
 case of a discrete spectrum.

\medskip
Does in make sense to use intermediate averages, e.g., the
$l^p$-averages
$$\lim\Bigl[\frac{1}{n}\sum_{k=0}^{n-1}\rho(T^k x,T^k y)^p\Bigr]^{\frac{1}{p}}=\hat \rho^p(x,y)$$
 for $p\in (1,\infty)$, instead of $l^1$?
Apparently, they do not lead to any new effects: these metrics behave in the same
way as the $l^1$-average metric. For instance, in the
above example, the $p$-average of the
metric determined by a two-block partition of the
unit circle is
 $\hat \rho^p(x,y)=\{m[A\Delta (A+x-y)]\}^{\frac{1}{p}}$.
For $p=\infty$ (i.e., the sup-metric), the picture is completely
different, as in other interpolation theories.
Thus in entropy theory, the use of average metrics substantially supplements
the classical considerations.

\subsection{Application of the discreteness criterion}

As noted above, the problem of determining whether or not the spectrum
of an automorphism is discrete,
is not at all simple. Theorems~5--7 provide convenient non-spectral
criteria for checking that the spectrum is not purely discrete; for
this, one should bound the entropy from below for one admissible
metric satisfying the conditions of Section~2 by a sequence that grows
arbitrarily slow with $n$.

One of the intriguing examples of automorphisms for which the
discreteness of the spectrum has not been neither proved nor disproved since
the 1980s is the Pascal automorphism. It was introduced by the author in
1980 (see \cite{V1,V2}) as an example of an {\it adic transformation},
and is defined as a natural transformation in the space of paths in the Pascal
graph regarded as a Bratteli--Vershik diagram with lexicographic
ordering of paths. One can give a short combinatorial description of
this transformation by encoding these paths with sequences of zeros and
ones and identifying the space of paths with the compact space
$X=\{0;1\}^{\infty}=\textbf{Z}_2$. Then the Pascal automorphism is
defined by the formula
$$T(\{1^i0^j1**\}=\{0^{j-1}1^{i+1}0**\});$$
here $i\geq 0$, $j>0$, and the domain of $T$ and $T^{-1}$ is the whole
$X$ except for the countable set of sequences having
finitely many zeros or ones. The most natural metric on $X$ is the
2-adic metric
$\rho(\{x_k\},\{y_k\})=2^{-n}$, where $n$ is the first digit with
$x_k\ne y_k$. This metric is admissible, and the Pascal automorphism satisfies
the Lipschitz condition almost everywhere. The orbits of this
automorphism coincide with the orbits of the action of the infinite
symmetric group. The Bernoulli measures are $T$-invariant. The
spectrum of the Pascal transformation was studied in the papers
\cite{Pe1,Pe2,Me,Rue}, where some interesting properties
were established  (e.g., it was proved that $T$ is loosely
Bernoulli, the complexity of $T$ was computed, etc.),
but the question about the type of
the spectrum  remains open. \footnote{{\it Note on the translation:} The answer is known now ---
the spectrum of Pascal automorphism is continuous.} In
\cite{V3,V4} it was conjectured that the study of the behavior of scaling
sequences may turn to be useful. The corresponding plan was carried out in
\cite{MM}, but in that paper a logarithmic lower bound was obtained on the
scaling sequence for the sup-metric, and not for the average metric;
this is not sufficient for the conclusion that the spectrum is not
discrete. Nevertheless, one may hope that the combinatorics developed
in \cite{MM} will help to prove that the scaling sequence is unbounded
also for the $\varepsilon$-entropy of the average metric, which, by
our theorem, would imply that the spectrum is not discrete. There are
many adic transformations similar to the Pascal automorphism for which
the same question is also of great interest. For example, if we replace the
Pascal graph with its multidimensional analog or the Young graph, we
will obtain automorphisms that supposedly have continuous spectra. As
observed above, in order to prove that there are no nontrivial eigenfunctions,
one should obtain a growing lower bound on the scaling sequence not
for one, but for all (or for some representative set of) bounded
admissible (semi)metrics.

\subsection{The dynamics of metrics}

Recall that the general approach that consists in studying the asymptotic behavior of
metrics is not exhausted by considering the asymptotics of the
$\varepsilon$-entropy of the average or supremum metric, i.e., does not
reduce to studying the growth  of scaling sequences; this is only
its simplest version. In fact, we consider the original measure space
$(X,{\frak A},\mu)$ with an action of an automorphism
$T$ (or a group of automorphisms $G$), fix an appropriate metric
$\rho$, and study the sequence of metric triples
$$(X,\rho^T_n,\mu), \quad\mbox{where } \rho_n^T(x,y)=\frac{1}{n}\sum_{k=0}^{n-1}\rho(T^k x,T^k y).
$$

The conjecture is that, for a fixed measure and a fixed automorphism (or
group of automorphisms), the asymptotic properties of this sequence of
metric triples do not depend (or weakly depend) on the choice of an individual
admissible metric from a wide class. These properties include not only
the scaling entropy, but also more complicated characteristics of the
sequence, say the mutual properties of several consecutive metric
triples. Since the classification of metric triples up to
measure-preserving isometry is known (see
\cite{Gr,VU}), one may hope to apply it to this problem. In this field
there are many traditional and nontraditional questions. For example,
what is the distribution of the fluctuations of the sequence of
average metrics, regarded as functions of two
variables on  $(X\times X,\mu\times \mu)$, as they converge to the constant metric (for weakly
mixing transformations, e.g., $K$-automorphisms)? What can be said about the
asymptotic properties of neighboring pairs
of metric triples (with indices $n$ and $n+1$)? Etc.

In conclusion, it is worth mentioning that the concept of scaling entropy
appeared in connection with the classification of filtrations
in \cite{V6} and was used in
\cite{VG}. In terms of the present paper, the scaling entropy for filtrations, i.e., decreasing
sequences of measurable partitions or $\sigma$-algebras, is the
scaling entropy for an action of a locally finite group such as
$\sum {\Bbb Z}/2$ instead of an action of $\Bbb Z$ considered here.
The definitions we have given for an action of
$\Bbb Z$ essentially coincide with those given in
\cite{V6} for locally finite groups.

\end{document}